\documentclass[a4paper,11pt,oneside]{article} 
\usepackage{amsfonts,amssymb}
\usepackage{color} \usepackage{mathrsfs} \let\mathcal\mathscr
\usepackage[all,ps,cmtip]{xy}
 \usepackage[top=1in, bottom=1in, left=1in, right=1in]{geometry}

\title{\sc Homological units}
\author{\sc Roland Abuaf \footnote{Imperial College London, South Kensington Campus, London SW7 2AZ, United-Kingdom and Institut des Hautes \'Etudes Scientifiques, 35 Route de Chartres, 91440 Bures-sur-Yvette, France. E-mail :\textit{r.abuaf@imperial.ac.uk, abuaf@ihes.fr}. Partially supported by the EPSRC programme grant EP/G06170X/1 and by the CARMIN programme grant.}}

\usepackage[dvips]{graphicx}

\usepackage{amsfonts,amssymb}
\usepackage{color} \usepackage{mathrsfs} \let\mathcal\mathscr

\usepackage[all,ps,cmtip]{xy}

\DeclareGraphicsExtensions{eps}
\DeclareGraphicsExtensions{pdf}

\usepackage[T1]{fontenc}
\usepackage{amssymb}

\usepackage{amsmath}
\usepackage{latexsym}

\usepackage[latin1]{inputenc}

\newtheorem{theo}{Theorem}[subsection]

\newtheorem{exem}[theo]{Example}

\newtheorem{quest}[theo]{Question}

\newtheorem{defi}[theo]{Definition}

\newtheorem{cor}[theo]{Corollary}
\newtheorem{conj}[theo]{Conjecture}

\def\DB{\mathrm{D^{b}}}
\def\OO{\mathcal{O}}

\def\DP{\mathrm{D^{perf}}}
\def\Ri{\mathrm{R^{i}}}
\def\R0{\mathrm{R^{0}}}

\def\HH{\mathrm{HH}}
\def\Hh{\mathcal{H}om}
\def\HHH{\mathrm{Hom}}
\def\LL{\mathrm{\textbf{L}}}

\def\RR{\mathrm{\textbf{R}}}

\def\OO{\mathcal{O}}
\def\EE{\mathcal{E}}

\def\D{\mathcal{D}}

\def\M{\mathcal{M}}

\def\C{\mathcal{C}}
\def\F{\mathcal{F}}
\def\GG{\mathrm{G}}
\def\G{\mathcal{G}}

\def\T{\mathcal{T}}

\def\X{\mathcal{X}}

\newenvironment{proof}
{
\noindent
\textit{\underline{Proof}} :\\
$\blacktriangleright\;$%
}
{\hspace{\stretch{1}}%
$\blacktriangleleft$}

\begin{document}

\maketitle

\begin{abstract}

We define and study the invariance properties of homological units. Some applications are given to the derived invariance of Hodge numbers. In particular, we prove that if $X$ and $Y$ are derived equivalent smooth projective varieties of dimension $4$ having the same $h^{1,1}$, then they have the same Hodge numbers. We also give a geometric interpretation of the conjectural invariance of homological units in terms of derived Jacobians.
\end{abstract}

\vspace{\stretch{1}}

\newpage

\begin{section}{Introduction}

Let $X$ and $Y$ be smooth projective varieties over the field of complex numbers. If $X$ and $Y$ are derived equivalent, there Hochschild homology groups are isomorphic \cite{orlov2}. The Hochschild-Kostant-Rosenberg Theorem \cite{markarian} then implies the isomorphism of vector spaces:

\begin{equation*}
\bigoplus_{p-q=k} H^q(X, \Omega_X^p) \simeq \bigoplus_{p-q=k} H^q(Y, \Omega_Y^p).
\end{equation*}

It is still unknown if all the graded pieces which appear in the above decomposition are in fact isomorphic. Namely, the following conjecture is folklore:

\begin{conj} \label{conjhodge}
Let $X$ and $Y$ be smooth projective varieties. Assume that $X$ and $Y$ are derived equivalent, then $X$ and $Y$ have the same Hodge numbers. 
\end{conj}
Note that if $X$ (or $Y$) has ample or anti-ample canonical bundle, a famous result of Bondal and Orlov states that derived equivalence of $X$ and $Y$ implies that both varieties are actually isomorphic \cite{BO}. Hence, conjecture \ref{conjhodge} is interesting when the sign of the canonical bundle of $X$ (and $Y$) is not definite. For instance if $X$ and $Y$ have trivial canonical bundle. 

\bigskip

In dimension $1$ and $2$ (and also in dimension $3$, if $X$ and $Y$ have trivial canonical bundle), this conjecture is an obvious consequence of the derived invariance of Hochschild homology, the Hochschild-Kostant-Rosenberg isomorphism and the Hodge symmetry. In dimension $3$, without any assumption on the canonical bundle, the conjecture is still true and has been settled by Popa and Schnell \cite{popa-schnell}.  If one assumes that the derived equivalence comes from a birational correspondence then the conjecture has been proved by Orlov \cite{orlov3}. In dimension $4$ or higher (and without the birational assumption), nothing is known, even if the canonical bundles of $X$ and $Y$ are trivial. 

\bigskip

Leaving aside this hard problem, one could be interested in a slightly less ambitious question:

\begin{quest}
Let $X$ and $Y$ be two smooth projective varieties. Assume that $X$ and $Y$ are derived equivalent. Denote by $\HH_{\bullet}(X)$ (resp. $\HH^{\bullet}(X)$), the Hochschild homology graded vector space (resp. Hochschild cohomology graded algebra) of $X$. 
\begin{itemize}
\item Are there non-trivial, canonically defined, sub-vector spaces of $\HH_{\bullet}(X)$ and $\HH_{\bullet}(Y)$ which are isomorphic?
\item Are there non-trivial, canonically defined, sub-algebras of $\HH^{\bullet}(X)$ and $\HH^{\bullet}(Y)$ which are isomorphic?
\end{itemize}
\end{quest}

The main goal of the present paper is to put forward some evidences that the graded algebras $H^{\bullet}(X, \OO_X)$ and $H^{\bullet}(Y,\OO_Y)$ are indeed isomorphic when $X$ and $Y$ are derived equivalent. Our main result shows that it is true for many examples of derived equivalences (see Theorem \ref{douce} and comments before for a more precise statement):

\begin{theo} \label{mainmain} Let $X$ and $Y$ be smooth projective varieties and $\Phi : \DB(X) \simeq \DB(Y)$ an equivalence of triangulated category. Assume that one of the following conditions holds:
\begin{enumerate}
\item The Fourier-Mukai kernel representing $\Phi$ is a (possibly shifted) generically pure vector bundle \footnote{In this statement,  a generically pure vector bundle on $X \times Y$ is an object on $X \times Y$ which is a pure vector bundle outside of a closed subset $Z \subset X \times Y$ such that $p(Z)$ and $q(Z)$ are strict closed subsets of $Y$ and $X$, where $q$ and $p$ are the obvious projections.} on $X \times Y$,

\item Both cohomology algebras $\bigoplus_{p=0}^{\dim X} H^p(X, \Omega^p_X), \bigoplus_{p=0}^{\dim Y} H^p(Y, \Omega^p_Y)$ are generated in degree $1$,

\item $X$ and $Y$ have dimension at most $4$,
\end{enumerate}
then there exists an isomorphism of graded algebras $H^{\bullet}(X, \OO_X) \simeq H^{\bullet}(Y, \OO_Y)$. 
\end{theo}

As a consequence of assertion 3 in Theorem \ref{mainmain}, we get:

\begin{cor}
Let $X$ and $Y$ be smooth derived equivalent smooth projective varieties. 
\begin{itemize}
\item if $X$ and $Y$ have dimension $3$, then all their Hodge numbers are the same.
\item if $X$ and $Y$ have dimension $4$ and the same $h^{1,1}$, then all their Hodge numbers are the same
\end{itemize}
\end{cor}

The sub-algebra $H^{\bullet}(X, \OO_X)$ plays a crucial role in \cite{abuaf-HK1, abuaf-HK2} for the definition of hyper-K\"ahler categories. In fact, I show in the first section of the present paper that an algebra having similar properties to $H^{\bullet}(X, \OO_X)$ can be defined for a large class of triangulated categories of geometric origin. I call them \textbf{homological units}. As mentioned above, even in the strict geometric case (i.e. the derived category of a smooth projective variety), I am not able to prove, in full generality, that these homological units are derived invariants. 

The derived invariance of the homological units can be made into a geometric statement using derived algebraic geometry. Namely, if $X$ is a smooth projective variety, I denote by $Pic^0(X)_{dg}$ the connected component of $\OO_X$ in the derived moduli stack of objects in $\DP(X)$ (see \cite{TV}). The derived stack $Pic^0(X)_{dg}$ is the derived Jacobian of $X$. I conjecture the following:

\begin{conj} \label{derivedjacobian}
Let $X$ and $Y$ be smooth projective varieties. Assume that $\DB(X) \simeq \DB(Y)$, then $Pic^0(X)_{dg}$ is isogeneous to $Pic^0(Y)_{dg}$.
\end{conj} 

Popa and Schnell proved the un-derived version of this conjecture (\cite{popa-schnell}) and they use it to prove the derived invariance of Hodge numbers in dimension $3$. Granted the derived invariance of homological units in dimension less or equal to 4, I can prove the following:

\begin{theo}
Conjecture \ref{derivedjacobian} is true in dimension less or equal to 4.
\end{theo}

\bigskip

\textbf{Acknowledgments:} I would like to thank Chris Brav and Laurent Manivel for interesting discussions and comments about the notion of homological units.

\end{section}

\begin{section}{Definition and derived invariance properties}

Let $X$ be a algebraic variety over $\mathbb{C}$ and let $\F \in \DB(X)$ be an object whose rank is not zero. Then the trace map:

\begin{equation*}
\mathrm{Tr} : \RR \Hh (\F,\F) \rightarrow \OO_X
\end{equation*}
splits and gives a splitting:

\begin{equation*}
 \HHH^{\bullet}(\F,\F) = H^{\bullet}(\OO_X) \oplus \HHH^{\bullet} (\F,\F)_0, 
\end{equation*}
where $\HHH^{\bullet} (\F,\F)_0$ is the vector space of trace-less endomorphisms. Hence, the algebra $H^{\bullet}(\OO_X)$ appears as a maximal direct factor of the endomorphisms algebra of any object in $\DB(X)$ which rank is not vanishing. The following is a categorical definition of what should be the avatar of $H^{\bullet}(\OO_X)$ in the non-commutative world.

\begin{defi} \label{homounit}
Let $\C$ be an abelian category with a non-trivial rank function and $\T$ be a full admissible subcategory in $\DB(\C)$. A graded algebra $\mathfrak{T}^{\bullet}$ is called a \textbf{homological unit} for $\T$ (with respect to $\C$), if $\mathfrak{T}^{\bullet}$ is maximal for the following properties : 
\begin{enumerate}

\item for any object $\F \in \T$, there exists a graded algebra morphism $\mathfrak{T}^{\bullet} \rightarrow  \HHH^{\bullet}(\F,\F)$ which is functorial in the following sense. Let $\F, \G \in \T$ and let $a \in \mathfrak{T}^{k}$ for some $k$. We denote by $a_{\F}$ (resp. $a_{\G}$) the image of $a$ in $\HHH^{k}(\F, \F)$ (resp. $\HHH^{k}(\G, \G)$). Then, for any morphism $\psi : \F \rightarrow \G$, there is a commutative diagram:
\begin{equation*}
\xymatrix{ \F \ar[rr]^{a_{\F}} \ar[dd]^{\psi} & &\F [k] \ar[dd]^{\psi[k]} \\
& &  \\
\G \ar[rr]^{a_{\G}} & & \G [k]}
\end{equation*}

\item for any $\F \in \T$ which rank (seen as an object in $\DB(\C)$) is not vanishing, the morphism $\mathfrak{T}^{\bullet} \rightarrow  \HHH^{\bullet}(\F,\F)$ is an injection of graded algebras, which splits as a morphism of vector spaces.

\end{enumerate}

With hypotheses as above, an object $\F \in \T$ is said to be \textup{unitary}, if $\HHH^{\bullet}(\F,\F) = \mathfrak{T}^{\bullet}$, where $\mathfrak{T}^{\bullet}$ is a homological unit for $\T$.

\end{defi}

Of course, one can not expect that all examples of homological units as defined above will be significant. In the main application of the present paper, one will look at $\C = Coh(X), Coh^{\GG}(X)$ or $Coh(X, \alpha)$, where $X$ is a smooth projective variety, $\GG$ an algebraic group acting on $X$, $\alpha$ a Brauer class on $X$ and the rank function will be the obvious one. However, it is well possible that many examples of homological units coming from Representation Theory will be discovered, so that it seems sensible to give a general definition that does not restrict to purely geometrical examples. 

Note also that the hypothesis of non-vanishing rank for the splitting is a technical condition which seems to be important. It would be very interesting to know if there are some non-trivial examples where the splitting occurs whatever the rank of the object.

\begin{exem}
\begin{enumerate}
\item Let $X$ be a smooth algebraic variety and $\alpha \in\mathrm{Br}(X)$, a class in the Brauer group of $X$. Consider $\C = Coh(X,\alpha)$, the category of coherent $\alpha$-twisted sheaves on $X$. One can define a rank function on $\C$ as being the rank of $\F$ when seen as an $\OO_X$-module. Then for any $\F \in \DB(\C)$, we have a trace map:

\begin{equation*}
\mathrm{Tr} : \RR \Hh_{\DB(\C)} (\F,\F) \rightarrow \OO_X
\end{equation*}

which splits when the rank of $\F$ is not zero. As a consequence, for all $\F \in \DP(\C)$, we have a graded algebra morphism:
\begin{equation*}
 H^{\bullet}(\OO_X) \rightarrow  \HHH^{\bullet}_{\DB(\C)}(\F,\F)
\end{equation*}
which splits (as a morphism of vector spaces) when the rank of $\F$ is not zero. The morphism $H^{\bullet}(\OO_X) \rightarrow  \HHH^{\bullet}_{\DB(\C)}(\F,\F)$ is given by $a \rightarrow id_{\F} \otimes a$, so that the functoriality property is clearly satisfied. Furthermore, if $L$ is a twisted line bundle in $\DB(Coh(X,\alpha))$, we have $\HHH^{\bullet}_{\DB(\C)}(L,L) = H^{\bullet}(\OO_X)$. Thus, $H^{\bullet}(\OO_X)$ is maximal for the properties required in Definition \ref{homounit} and it is a homological unit for $\C$.

\item Let $X$ be a smooth algebraic variety and $\GG$ be a finite group acting on $X$. For any $\F \in \DB(Coh^{\GG}(X))$, the trace map $\mathrm{Tr} : \RR \Hh (\F, \F) \rightarrow \OO_X$ is $\GG$-equivariant and it is split if the rank of $\F$ is non-zero. Hence, for all $\F \in \DB(Coh^{\GG}(X)) $, we have a graded algebra morphism:

\begin{equation*}
 H^{\bullet}(\OO_X)^{\GG} \rightarrow \HHH^{\bullet}_{\DB(Coh^{\GG}(X))}(\F, \F), 
\end{equation*}
which splits (as a morphism of vector spaces) when the rank of $\F$ is not zero. The morphism $H^{\bullet}(\OO_X)^{\GG} \rightarrow  \HHH^{\bullet}(\F,\F)$ is again given by $a \rightarrow id_{\F} \otimes a$, so that the functoriality property is also satisfied. Furthermore, if $L$ is a $\GG$-invariant line bundle on $X$, we have $\HHH^{\bullet}_{\DB(\C)}(L,L) = H^{\bullet}(\OO_X)^{\GG}$. Hence, the algebra $ H^{\bullet}(\OO_X)$ is maximal for the properties required in Definition \ref{homounit} and it is a homological unit for $\C$. This readily generalizes for any smooth Deligne-Mumford stack. Namely, if $\X$ is a smooth Deligne-Mumford stack, then $H^{\bullet}(\OO_{\X})$ is a homological unit for $\DB(\X)$. Note that all line bundles on $\X$ are unitary objects.
\end{enumerate} 
\end{exem}

\bigskip

By the maximality property required in the definition, the homological unit is unique (with respect to the embedding in the derived category of an abelian category) as soon as there exists a unitary object in $\T$ (and so being unitary is non-ambiguous with respect to the homological unit). This is for instance the case when $\C = Coh(X, \alpha)$ or $\C = Coh^{\GG}(X)$, as any (twisted/equivariant) line bundle is unitary. As for the independence with respect to the embedding in the derived category of an abelian category, the question seems to be more delicate.

\bigskip

Let $\C_1$ and $\C_2$ be two abelian categories with rank functions and let $\T$ be a full admissible subcategory of both $\DB(\C_1)$ and $\DB(\C_2)$. Let $\mathfrak{T}_1$ a homological unit for $\T$ with respect to $\C_1$ and $\mathfrak{T}_2$ a homological unit for $\T$ with respect to $\C_2$. Assume that there exists a unitary object in $\T$ with respect to $\C_1$ whose rank is not vanishing when seen as an object in $\DB(\C_2)$. Then, by definition, the algebra $\mathfrak{T}_2$ is a graded algebra direct summand of $\mathfrak{T}_1$. Assume furthermore that there exists a unitary object in $\T$ with respect to $\C_2$ whose rank is not vanishing as an object of $\DB(\C_1)$. Then there is graded algebra isomorphism $\mathfrak{T}_1 \simeq \mathfrak{T}_2$. As a consequence, the question of the independence of homological units with respect to the embedding in the derived category of an abelian category can be reduced to the following:

\begin{quest} \label{questunit}
Let $\C_1$ and $\C_2$ be two an abelian categories with rank functions and let $\T$ be a triangulated category which can be embedded as an admissible full subcategory of both $\DB(\C_1)$ and $\DB(\C_2)$. When is there a unitary object in $\T$ with respect to $\C_1$ (resp. $\C_2$) whose rank is not vanishing when seen as an object in $\DB(\C_2)$ (resp. $\DB(\C_1)$)?
\end{quest}

In the special case where $\T \simeq \DB(\C_1) \simeq \DB(C_2)$ and $\C_1$ and $\C_2$ are the categories of coherent sheaves on algebraic varieties, I can give a partial answer to \ref{questunit};

\begin{theo} \label{douce} Let $X$ and $Y$ be smooth projective varieties and $\Phi : \DB(X) \simeq \DB(Y)$ an equivalence of triangulated category. Assume that one of the following conditions holds:
\begin{enumerate}
\item The Fourier-Mukai kernel representing $\Phi$ a (possibly shifted) generically pure vector bundle on $X \times Y$,

\item The $i$-th graded components of the Chern character of the kernel of $\Phi$ in $H^{\bullet}(X \times Y, \mathbb{C})$ are zero for $i=0 \cdots 2 \dim X-1$,

\item Both cohomology algebras $\bigoplus_{p=0}^{\dim X} H^p(X, \Omega^p_X), \bigoplus_{p=0}^{\dim Y} H^p(Y, \Omega^p_Y)$ are generated in degree $1$,

\item $X=Y$ is smooth and $\Phi$ is a spherical twist or a $\mathbb{P}^n$-twist,

\item $X$ and $Y$ have dimension at most $4$,
\end{enumerate}
then there exists a unitary object in $\F \in \DB(X)$ such that the rank of $\Phi(\F)$ is not zero and vice-versa. In particular, the homological unit of $\DB(X)$ is isomorphic to the homological unit of $\DB(Y)$.

\end{theo}

In the first condition, a generically pure vector bundle on $X \times Y$ is an object on $X \times Y$ which is a pure vector bundle outside of a closed subset $Z \subset X \times Y$ such that $p(Z)$ and $q(Z)$ are strict closed subsets of $Y$ and $X$, where $p$ and $q$ are the obvious projections.

\bigskip

\begin{proof}

1. Let me start proving that the first condition is sufficient for the existence of a unitary object $\F \in \DB(X)$ (resp. $\G \in \DB(Y)$) such that the rank of $\Phi(\F)$ (resp. $\Phi^{-1}(\G)$) is not zero. Let $\EE$ be an object on $X \times Y$ which represents the kernel of $\Phi$. By hypothesis, the object $\EE$ is a generically pure (possibly shifted) vector bundle. So, up to shift, we can assume that $\EE$ is a generically pure vector bundle concentrated in degree $0$. We have a diagram:

\begin{equation*}
\xymatrix{ & &  \ar[lldd]_{q} X \times Y \ar[rrdd]^{p} & &  \\
& & & & \\
X & & & & Y}
\end{equation*}

and $\Phi$ is given by $\RR p_*(\LL q^*(?) \otimes \EE) : \DB(X) \rightarrow \DB(Y)$. Let $\OO_{X}(1)$ be an ample line bundle on $X$. We will prove that there exists $k>0$ such that $\RR p_*(\LL q^*( \OO_X(k)) \otimes \EE)$ has non-vanishing rank, as an object in $\DB(Y)$. We proceed by contradiction. Assume that for all $k>0$, the object $\RR p_*(\LL q^*( \OO_X(k)) \otimes \EE)$ has rank $0$. Since $q^* \OO_{X}(1)$ is relatively ample for $p$ and $\EE$ is a generically pure vector bundle, we have:

\begin{equation*}
\operatorname{rank}(\Ri p_*(q^* \OO_{X}(k) \otimes \EE)) = 0,
\end{equation*}
for all $i>0$ and all $k$ big enough. Hence the hypothesis that the object $\RR p_*(\LL q^*\OO_X(k) \otimes \EE)$ has rank $0$ for all $k>0$ implies that the object $p_* (q^* \OO_{X}(k) \otimes \EE)$ has rank $0$, for all $k$ big enough. Since $\EE$ is a generically pure vector bundle, this shows that the support of $p_* (q^* \OO_{X}(k) \otimes \EE)$ is a strict subvariety of $Y$, for all $k$ big enough. Let $U$ be a strict open of $Y$ which contains $\mathrm{Supp} \left( p_* (q^* \OO_{X}(k) \otimes \EE) \right)$ for all $k$ big enough. Since $q^* \OO_{X}(1)$ is relatively ample for $p$, the adjunction morphism:

\begin{equation*}
p^* p_* (q^* \OO_{X}(k) \otimes \EE) \rightarrow q^* \OO_{X}(k) \otimes \EE
\end{equation*}
is surjective for $k$ big enough. This implies that the support of $\EE$ is contained in $p^{-1}(U)$. In particular, the image of $\Phi$ is included in $\D^b(U)$ and $\Phi$ is not an equivalence. This is a contradiction.

Hence we have proved that there exists $k>0$ such that the rank of $\RR p_* (q^* \OO_{X}(k) \otimes \EE)$ is not zero. 
\bigskip

To prove that there exists a unitary object $\G$ in $\DB(Y)$ such that the rank of $\Phi^{-1}(\G) \in \DB(X)$ is not zero, we notice that $\Phi^{-1}$ is given by : $\RR q_* (\LL p^*(?) \otimes \RR \Hh(\EE,p^* \omega_Y))$. But $\EE$ is a (shifted) generically pure vector bundle on $X \times Y$, so that $\RR \Hh(\EE,p^* \omega_Y))$ is also a (shifted) generically pure vector bundle. As a consequence, the functor : $\RR q_* (\LL p^*(?) \otimes \RR \Hh(\EE,p^* \omega_Y)$ is an equivalence going from $\DB(Y)$ to $\DB(X)$ whose kernel is a (possibly shifted) generically pure vector bundle. The same argument as before shows that there exists $\G \in \DB(Y)$ such that the rank of $\Phi^{-1}(\G) \in \DB(X)$ is not zero.

\bigskip
\bigskip

2. We focus now on the second condition. Let $\EE$ be the kernel of a Fourier-Mukai equivalence between $X$ and $Y$. We will prove that the image of $\OO_X$ has non-vanishing rank. We again proceed by absurd. Assume that the rank of $\RR p_* \EE$ is zero. Since $Y$ is smooth and projective, the object $\RR p_* \EE$ is quasi-isomorphic to a bounded complex of vector bundles. Hence the condition $\mathrm{rank}(\RR p_* \EE) = 0$ is equivalent to $\chi(\RR p_* \EE \otimes \mathbb{C}(y)) =0$, for all $y \in Y$. By the projection formula and the Leray spectral sequence, this implies that for any $y \in Y$, we have:

\begin{equation*}
\chi(\LL j_y^* \EE)=0
\end{equation*}
where $j_y : X \times y \hookrightarrow X \times Y$ is the embedding of the fiber of $p$ over $y$ in $X \times Y$. By the Grothendieck-Rieman-Roch formula, this vanishing is equivalent ot the vanishing:

\begin{equation*}
\int_{X} \mathrm{ch}(\LL j_y^* \EE).td(X) = 0.
\end{equation*}

But the $i$-th graded components of $\mathrm{ch}(\EE) \in H^{\bullet}(X \times Y, \mathbb{C})$ are zero for $i=0 \cdots 2\dim X-1$ and have the compatibility condition : $j_y^* \mathrm{ch}(\EE)= \mathrm{ch}(\LL j_y^* \EE)$. As a consequence, we find that the $i$-th graded components of $\mathrm{ch}(\LL j_y^* \EE)$ are zero for $i=0 \cdots 2 \dim X -1$. The vanishing $\int_{X} \mathrm{ch}(\LL j_y^* \EE).td(X) = 0$ hence implies that $\mathrm{ch}(\LL j_y^* \EE)_{2 \dim X} = 0$, which finally proves that $\mathrm{ch}(\LL j_y^* \EE) = 0$, for all $y \in Y$. Since $ \LL j_y^* \EE \simeq \Phi(\mathbb{C}(y))$ and $\Phi$ is an equivalence, this proves that $\mathrm{ch}(\mathbb{C}(y)) = 0$. Contradiction.

\bigskip

The kernel giving the inverse of $\Phi$ has the same support as $\EE$, so that the above argument also applies to $\Phi^{-1}$.

\bigskip
\bigskip
 
3. Let me turn to the third condition. We want to prove that there exists $L \in \mathrm{Pic}(X)$, such that the rank of $\Phi(L)$ as an object in $\DB(Y)$ is not zero. One again we proceed by absurd. Assume that for all $L \in \mathrm{Pic}(X)$, the object $\phi(L)$ has rank $0$. Since $Y$ is smooth projective, the vanishing of $\mathrm{rank}(\Phi(L))$ is equivalent to $\chi(\Phi(L) \otimes \mathbb{C}(y)) = 0$, for any point $y \in Y$. This implies that $\chi(L \otimes \Phi^{-1}(\mathbb{C}(y)) = 0$, for any $y \in Y$ and any $L \in \mathrm{Pic}(X)$. In particular, for any $p \geq 1$, any $L_1,\cdots,L_p \in \mathrm{Pic}(X)$ and any $m_1, \cdots, m_p \in \mathbb{Z}$, we have:

\begin{equation*}
\chi(L_1^{\otimes m_1} \otimes \cdots \otimes L_p^{\otimes m_p} \otimes \Phi^{-1}(\mathbb{C}(y))) = 0.
\end{equation*}
By the Grothendieck-Riemann-Roch formula, this is equivalent to:

\begin{equation*}
\int_{X} \mathrm{ch}(L_1^{\otimes m_1} \otimes \cdots \otimes L_p^{\otimes m_p}).\mathrm{ch}(\Phi^{-1}(\mathbb{C}(y))).td(X) = 0,
\end{equation*}

for all $p \geq 1$, all $L_1,\cdots,L_p \in \mathrm{Pic(X)}$, all $m_1,\cdots,m_p \in \mathbb{Z}$ and all $y \in Y$. Let me prove that this implies $\mathrm{ch}(\Phi^{-1}(\mathbb{C}(y))) = 0$. Indeed, if $\mathrm{ch}(L_1^{\otimes m_1} \otimes \cdots \otimes L_p^{\otimes m_p})_k$ is the component of $\mathrm{ch}(L_1^{\otimes m_1} \otimes \cdots \otimes L_p^{\otimes m_p})$ in $H^k(X,\mathbb{C})$, then we have:
\begin{equation*}
\begin{split}
\mathrm{ch}(L_1^{\otimes m_1} \otimes \cdots \otimes L_p^{\otimes m_p})_k & =  (m_1c_1(L_1)+ \cdots +m_pc_1(L_p))^k \\
& = \sum_{k_1+ \cdots + k_p =k}  \frac{k!}{k_1! \cdots k_p!} m_1^{k_1} \cdots m_p^{k_p} c_1(L_1)^{k_1} \cdots c_1(L_p)^{k_p}
\end{split}
\end{equation*}

As a consequence, the equation:
\begin{equation*}
\int_{X} \mathrm{ch}(L_1^{\otimes m_1} \otimes \cdots \otimes L_p^{\otimes m_p}).\mathrm{ch}(\Phi^{-1}(\mathbb{C}(y))).td(X) = 0,
\end{equation*}

reads:

\begin{equation*}
\sum_{k=0}^{\dim X} \sum_{k_1+ \cdots + k_p =k}  \frac{k!}{k_1! \cdots k_p!} m_1^{k_1} \cdots m_p^{k_p} c_1(L_1)^{k_1} \cdots c_1(L_p)^{k_p}. \mathrm{ch}(\left(\Phi^{-1}(\mathbb{C}(y))).td(X) \right)_{2\dim X - 2k} = 0,
\end{equation*}
for all $m_1, \cdots m_p \in \mathbb{Z}$. In particular, we get:

\begin{equation*}
 c_1(L_1)^{k_1} \cdots c_1(L_p)^{k_p}. \left( \mathrm{ch}(\Phi^{-1}(\mathbb{C}(y))).td(X) \right)_{2\dim X - 2k}=0,
\end{equation*}

for all $L_1, \cdots L_p \in \mathrm{Pic}(X)$, all $k_1, \cdots k_p \in \mathbb{N}$ such that $k_1+ \cdots k_p = k$ and all $0 \leq k \leq \dim X$. By hypothesis, the cohomology algebra $\bigoplus_{p=0}^{\dim X} H^p(X, \Omega^p_X)$ is generated in degree $1$, hence by Lefschetz $(1,1)$ Theorem, we know that numerical and homological equivalences coincide on $X$. We deduce that $\mathrm{ch}(\Phi^{-1}(\mathbb{C}(y))).td(X) =0$. This demonstrates that $\mathrm{ch}(\Phi^{-1}(\mathbb{C}(y))) = 0$ as the Todd class is invertible. From this vanishing, we find that $\mathrm{ch}(\mathbb{C}(y)) = 0$ as $\Phi$ is an equivalence. This is absurd. The hypotheses being symmetric on $X$ and $Y$, we deduce that there also exists a unitary object in $\G \in \DB(Y)$ such that $\Phi^{-1}(\G)$ has non-vanishing rank.
\bigskip
\bigskip

4. The proof of the statement when the fourth condition is satisfied can be deduced from the previous statements. Indeed, a $\mathbb{P}^n$-twist is a Fourier-Mukai transform whose kernel $\EE$ is given by the exact triangle:

\begin{equation*}
(E^{*} \boxtimes E[-2] \rightarrow E \boxtimes E) \rightarrow \OO_{\Delta} \rightarrow \EE,
\end{equation*}
where $E$ is an object in $X$ whose endomorphisms algebra is isomorphic to $\mathbb{C}[t^2]/(t^{n+1})$ with $t$ in degree $2$ and $\Delta \subset X \times X$ is the diagonal. Hence, the Chern character of $\EE$ is equal to the Chern character of $\OO_{\Delta}$ and we are back to the second condition.

\bigskip

Assume $\Phi$ is a spherical twist. Let $\EE$ be the kernel of $\Phi$. We have an exact triangle:

\begin{equation*}
E^* \boxtimes E \rightarrow \OO_{\Delta} \rightarrow \EE.
\end{equation*}  

Thus, for all $L \in \mathrm{Pic}(X)$, we have:

\begin{equation*}
\RR \Gamma (E^* \otimes L) \otimes E \rightarrow L \rightarrow \Phi(L).
\end{equation*}

If $\mathrm{rank}(E) = 0$, then we have $\mathrm{rank}(\Phi(L))=1$ and we are done. Assume that $\mathrm{rank}(E) \neq 0$. We will prove that for $L$ ample, the absolute value of the rank of the complex of vector spaces $\RR \Gamma (E^* \otimes L^{\otimes m})$ gets arbitrarily big, which demonstrates that the rank of $\Phi(L^{\otimes m})$ can not be always $0$. We proceed by absurd. Assume that $\mathrm{rank}( \RR \Gamma (E^* \otimes L^{\otimes m}))$ remains bounded. For all $m \in \mathbb{Z}$, we have:

\begin{equation*}
\begin{split}
\mathrm{rank}( \RR \Gamma (E^* \otimes L^{\otimes m})) & = \chi(E^* \otimes L^{\otimes m}) \\
& = \int_{X} \mathrm{ch}(E^*).\mathrm{ch}(L^{\otimes m}).td(X),
\end{split}
\end{equation*}
where the second equality is the Grothendieck-Riemann-Roch formula. But the $k$-th graded component of $\mathrm{ch}(L ^{\otimes m})$ in $H^{\bullet}(X,\mathbb{C})$ is equal to $m^k c_1(L)^k$. Hence, we have:

\begin{equation*}
\int_{X} \mathrm{ch}(E^*).\mathrm{ch}(L^{\otimes m}).td(X) = \sum_{k=0}^{2 \dim X} m^k c_1(L)^k. (\mathrm{ch}(E^*).td(X))_{2 \dim X-2k}. 
\end{equation*}

The boundedness of $\mathrm{rank}( \RR \Gamma (E^* \otimes L^{\otimes m})) $ then implies that $c_1(L)^{\dim X}.(\mathrm{ch}(E^*).td(X))_{0}=0$. But $L$ is ample, so that $c_1(L)^{\dim X} \neq 0$. As a consequence, we find $\mathrm{ch}(E^*).td(X)_{0} = 0$. This is a contradiction as $\mathrm{ch}(E^*).td(X)_{0} = \mathrm{rank}(E^*) = \mathrm{rank}(E)$.

\bigskip
\bigskip

5. We assume that $\dim X = \dim Y = 4$, the cases $\dim X = \dim Y \leq 3$ being anything if easier. Let $\Phi : X \rightarrow Y$ be the equivalence and let $\EE$ be an object on $X \times Y$ representing the kernel of $\Phi$. Assume that for all $L \in Pic(X)$, the rank of $\Phi(L)$ is zero. As in the proof of assertion 3, we get:

\begin{equation*}
 c_1(L_1)^{k_1} \cdots c_1(L_p)^{k_p}. \left( \mathrm{ch}(\Phi^{-1}(\mathbb{C}(y)).td(X)) \right)_{2\dim X - 2k}=0,
\end{equation*}
for all $L_1, \cdots L_p \in \mathrm{Pic}(X)$, all $k_1, \cdots k_p \in \mathbb{N}$ such that $k_1+ \cdots k_p = k$, all $0 \leq k \leq 4$ and all $y \in Y$. Since homological equivalence and numerical equivalence coincide for curves and divisors, we get $\left( \mathrm{ch}(\Phi^{-1}(\mathbb{C}(y))).td(X) \right)_{2k} = 0$ for $k=0,1,3,4$. Let us prove that $\left( \mathrm{ch}(\Phi^{-1}(\mathbb{C}(y))).td(X) \right)_{4}$ also vanishes. We proceed by contradiction. By the above equation, we know that $\left(\mathrm{ch}( \Phi^{-1}(\mathbb{C}(y))).td(X) \right)_{4}$ is in the primitive cohomology of $Y$. If $\left( \mathrm{ch}(\Phi^{-1}(\mathbb{C}(y))).td(X) \right)_{4} \neq 0$, the Hodge-Riemann bilinear relations imply:

\begin{equation*}
\left(\mathrm{ch}( \Phi^{-1}(\mathbb{C}(y))).td(X) \right)_{4}.\left(\mathrm{ch}( \Phi^{-1}(\mathbb{C}(y))).td(X) \right)_{4} \neq 0.
\end{equation*}
But $\left( \mathrm{ch}(\Phi^{-1}(\mathbb{C}(y))).td(X) \right)_{2k} = 0$ for $k=0,1,3,4$, so that 
$$\mathrm{ch}( \Phi^{-1}(\mathbb{C}(y)))  = \mathrm{ch}( \Phi^{-1}(\mathbb{C}(y))) .td(X).td(X)^{-1}$$
has non-vanishing components in $H^{2k}(X,\mathbb{C})$ only for $k \geq 2$ and its component in $H^4(X, \mathbb{C})$ is $\mathrm{ch}( \Phi^{-1}(\mathbb{C}(y))).td(X)_{4}$. We deduce that:
 \begin{equation*}
\begin{split}
& \int_X \mathrm{ch}(\RR \Hh (\Phi^{-1}(\mathbb{C}(y)), \OO_Y)). \mathrm{ch}(\Phi^{-1}(\mathbb{C}(y))).td(X) \\
= & \left(\mathrm{ch}( \Phi^{-1}(\mathbb{C}(y))).td(X) \right)_{4}.\left(\mathrm{ch}( \Phi^{-1}(\mathbb{C}(y))).td(X) \right)_{4}  \\
\neq &  0
\end{split}
\end{equation*}

As $\Phi$ is an equivalence, we have:

\begin{equation*}
\begin{split}
& \int_X \mathrm{ch}(\RR \Hh (\Phi^{-1}(\mathbb{C}(y)), \OO_Y)). \mathrm{ch}(\Phi^{-1}(\mathbb{C}(y))).td(X) \\
=& \int_Y \mathrm{ch}(\RR \Hh (\mathbb{C}(y), \OO_Y)).\mathrm{ch}(\mathbb{C}(y)).td(Y) \\
= & 0.
\end{split}
\end{equation*}
This is a contradiction. We deduce that $\left( \mathrm{ch}(\Phi^{-1}(\mathbb{C}(y))).td(X) \right)_{4}=0$. As a consequence, we have $\mathrm{ch}(\Phi^{-1}(\mathbb{C}(y)))=0$, which is impossible since $\Phi$ induces a bijection between $H^*(X,\mathbb{C})$ and $H^*(Y,\mathbb{C})$. Thus, there exists a line bundle $L$ on $X$ such that the rank of $\Phi(L)$ is non-zero.

\end{proof}

A few comments should be made about this result. Many examples of derived equivalence come from Mukai theory, where one variety is a fine moduli space of objects on the other and the kernel giving the equivalence is a universal bundle on the product (see \cite{mukai-abelian, mukai-K3}). In this case, the kernel satisfies condition $1$ of Theorem \ref{douce}.
\bigskip

If $X$ and $Y$ are birational, then all known examples of derived equivalences come from kernel which are supported in half codimension (\cite{BO, bri, chen-flop, nami-mukai1, nami-mukai2, kawa-mukai, cautis-flop}). Thus, these objects satisfy condition $2$ of Proposition \ref{douce}. It follows that all geometric derived equivalences I am aware of fall under condition $1,2$ or $4$ of Theorem \ref{douce}. It suggests that this proposition should be true in a much larger context.

\bigskip

Let us mention an amusing corollary of assertion 5 In Theorem \ref{douce}:

\begin{cor}
Let $X$ and $Y$ be smooth derived equivalent smooth projective varieties. 
\begin{itemize}
\item if $X$ and $Y$ have dimension $3$, then all their Hodge numbers are the same.
\item if $X$ and $Y$ have dimension $4$ and the same $h^{1,1}$, then all their Hodge numbers are the same
\end{itemize}
\end{cor}

\begin{proof}
Assertion 5 of Proposition \ref{douce} shows that $h^i(\OO_X) = h^i(\OO_Y)$ for all $i$ if $X$ and $Y$ have dimension less or equal to $4$. In dimension $3$, the Hochschild-Kostant-Rosenber isomorphism and the derived invariance of Hoschshcild homology immediately imply the result. In dimension $4$, inspection of the Hodge numbers which appear in th HKR decomposition shows that, under the hypothesis $h^i(\OO_X) = h^i(\OO_Y)$ for all $i$, the numbers which could be non-invariant are $h^{1,1}$ and $h^{2,2}$ (but the sum $2h^{1,1} + h^{2,2}$ is invariant). Hence if one assumes that $h^{1,1}(X) = h^{1,1}(Y)$, we have the equality of all Hodge numbers.
\end{proof}

\bigskip

In dimension $3$, this result was already proved in \cite{popa-schnell}

\bigskip

As far as the invariance of homological units for hyper-K\"ahler manifolds is concerned (without any rank condition on the image of unitary objects), let us notice the following Theorem of Huybrechts and Nieper-Wisskirchen \cite{Huy-Nieper}:

\begin{theo}[Huybrechts-Nieper-Wisskirchen] \label{huy-nie}
Let $X$ and $Y$ be two smooth projective derived equivalent varieties. Then $X$ is a hyper-K\"ahler manifold if and only if $Y$ is a hyper-K\"ahler manifold. In particular, if $X$ is hyper-K\"ahler, the homological units of $\DB(X)$ and $\DB(Y)$ are isomorphic.
\end{theo}

If the Fourier-Mukai functor giving the equivalence satisfies one of the hypotheses of proposition \ref{douce}, then the above theorem is a direct consequence of proposition \ref{douce} and of proposition $A.1$ in the appendix of \cite{Huy-Nieper}. As far as I am aware, all known derived equivalences between hyper-K\"ahler manifolds satisfy either hypothesis $1,2$ or $4$ of proposition \ref{douce}. In the general case, it seems however that some advanced techniques are required for the proof of Theorem \ref{huy-nie}

\end{section}

\begin{section}{Connection with derived Jacobians}
I will conclude this paper by describing a geometric interpretation of the conjectural invariance of homological units. Let $X$ and $Y$ be smooth projective varieties. Rouquier \cite{Rouquier} proved that if $X$ and $Y$ are derived equivalent, then $Aut^0(X) \times Pic^0(X) \simeq Aut^0(Y) \times Pic^0(Y)$ as algebraic groups. It was furthermore shown in \cite{popa-schnell} that the Lie algebra of $Pic^0(X)$ is isomorphic to the Lie algebra of $Pic^0(Y)$ (thus giving an isomorphism $H^1(\OO_Y) \simeq H^1(\OO_X)$). 

\bigskip

It seems likely that the derived invariance of the Lie algebra of the Jacobian should be true at the derived level. Indeed, the Jacobian of $X$ can be seen as the $0$-th truncated part of the connected component of $\OO_X$  in the derived moduli stack of objects in $\DP(X)$ \cite{TV}. We denote by $Pic^0(X)_{dg}$ this connected component and we call it the derived Jacobian of $X$. Extending the techniques of \cite{popa-schnell} in the derived setting, one could perhaps prove that if $X$ and $Y$ are derived equivalent, then the differential graded Lie algebra of $Pic^0(X)_{dg}$ is quasi-isomorphic to the differential graded Lie algebra of $Pic^0(Y)_{dg}$. Since the former (resp. the latter) is naturally isomorphic to $\mathrm{Ext}^{\bullet}(\OO_X, \OO_X)[-1]$ (resp. $\mathrm{Ext}^{\bullet}(\OO_Y, \OO_Y)[-1]$), the invariance of the homological units would follow immediately. 

\begin{conj}
Let $X$ and $Y$ be smooth projective variety. Assume that $X$ and $Y$ are derived equivalent. Then,$Pic^0(X)_{dg}$ and $Pic^0(Y)_{dg}$ are isogeneous as derived Jacobians.
\end{conj}   
Note that a similar idea has been developed (and applied successfully) by Keller \cite{keller2} in the affine case in order to prove the derived invariance of Hochschild cohomology endowed with its Gerstenhaber bracket. 

Before going any further, let  me clarify what is an isogeny of derived Jacobians. If $f : X^{\bullet} \rightarrow Y^{\bullet}$ is a map of derived stacks, with appropriate finiteness conditions (which will be automatically satisfied in our context), we get a cotangent map of DG-algebras between their cotangent complexes : $\mathbb{L}_{f} : \mathbb{L}_{X^{\bullet}} \rightarrow \mathbb{L}_{Y^{\bullet}}$. We say that $f$ is \textit{\'etale}, if $\mathbb{L}_{f}$ is a quasi-isomorphism of DG-algebra \cite{toen3}. Now, we say that a morphism between two derived Jacobians $Pic^0(X)_{dg}$ and $Pic^0(Y)_{dg}$ is an isogeny if the $0$-th truncated part of the morphism is an isogeny of abelian varieties and if the whole morphism is \'etale in the derived sense. Using the techniques developed in the present paper, I can prove the following:

\begin{theo} \label{jacjac}
Let $X$ and $Y$ be smooth projective varieties. Assume that $\DB(X) \simeq \DB(Y)$. Then,
\begin{enumerate}
\item the Picard varieties $Pic^0(X)$ and $Pic^0(Y)$ are isogeneous,

\item if $\dim X \leq 4$, then $Pic^0(X)_{dg}$ and $Pic^0(Y)_{dg}$ are isogeneous as derived Jacobians.
\end{enumerate}
\end{theo}

As mentioned before, the first part of this result was proved by Popa and Schnell in \cite{popa-schnell}. The proof I will give here seems to be very elementary and could perhaps be generalized in the derived setting. Furthermore, the arguments raised to prove $(1)$ will be used to prove $(2)$.

\bigskip

\begin{proof}

Let me start with the proof of $(1)$. Let $\M_X$ (resp. $\M_Y$) be the moduli stack of perfect complexes $\C$ on $X$ (resp. $Y$) satisfying $\mathrm{Ext}^i(\C,\C) = 0$ for $i < 0$ and $\mathrm{Hom}(\C,\C) = \mathbb{C}$. It was proved in \cite{Inaba} that this stack is an algebraic space.

Let $\Phi : \DB(X) \rightarrow \DB(Y)$ be an equivalence, the induced map (which we still denote by $\Phi$) between $\M_X$ and $\M_Y$ is an isomorphism. Let $\OO_X(1)$ be an ample line bundle on $X$. Since the sequence $\{ \OO_{X}(n) \}_{n \in \mathbb{N}}$ generates $\DB(X)$, we can find an integer $m$ such that one of the cohomology sheaves of $\Phi(\OO_{X}(m))$ has support equal to $Y$. We denote by $Pic^{0,m}(X)$ the canonical torsor under $Pic^0(X)$ with base point $\OO_{X}(m)$. The space $\Phi(Pic^{0,m})$ is a smooth variety which is a connected component of $\M_Y$ and is isomorphic to $Pic^{0}(X)$ (as an algebraic variety without group structure).

\bigskip

Let $Pic^0(Y)$ act on $\Phi(Pic^{0,m}(X))$ by tensor product. This action is well defined. Indeed, if $\C$ is any point in $\Phi(Pic^{0,m}(X))$, then $Pic^0(Y).\C$ is a connected subvariety of $\M_Y$ which contains $\C$. Since $\Phi(Pic^{0,m}(X))$ is the connected component of $\M_Y$ containing $\C$, we see that $Pic^0(Y).\C \subset \Phi(Pic^{0,m}(X))$.

Let $Pic^0(Y)_{\Phi(\OO_{X}(m))}$ be the stabilizer of $\Phi(\OO_{X}(m))$ with respect to the action of $Pic^0(Y)$ on $\Phi(Pic^{0,m}(X))$. We will prove that $Pic^0(Y)_{\Phi(\OO_{X}(m))}$ is a finite subgroup of $Pic^0(Y)$. Indeed, let $L \in Pic^0(Y)_{\Phi(\OO_{X}(m))}$. By definition, $L \otimes \Phi(\OO_{X}(m))$ is quasi-isomorphic to $\Phi(\OO_{X}(m))$. This implies that for all $j \in \mathbb{Z}$, we have:

\begin{equation*}
\mathcal{H}^j(\Phi(\OO_{X}(m))) \otimes L \simeq \mathcal{H}^{j}(\Phi(\OO_{X}(m))).
\end{equation*} 
But we know there exists $j_0 \in \mathbb{Z}$ such that the support of $\mathcal{H}^{j_0}(\Phi(\OO_{X}(m)))$ is all $Y$. Hence the rank of $\mathcal{H}^{j_0}(\Phi(\OO_{X}(m)))$ as a sheaf on $Y$ is non-zero. Taking determinant in the above equation, we get:

\begin{equation*}
det \mathcal{H}^{j_0}(\Phi(\OO_{X}(m))) \otimes L^{\mathrm{rank}(\mathcal{H}^{j_0}(\Phi(\OO_{X}(m))))} \simeq det \mathcal{H}^{j_0}(\Phi(\OO_{X}(m))),
\end{equation*}
so that $L^{\mathrm{rank}(\mathcal{H}^{j_0}(\Phi(\OO_{X}(m))))} = \OO_Y$. Since $\mathrm{rank}(\mathcal{H}^{j_0}(\Phi(\OO_{X}(m)))) \neq 0$, this equation has a finite number of solutions in $Pic^0(Y)$. This proves that $Pic^0(Y)_{\Phi(\OO_{X}(m))}$ is finite. The quotient group $Pic^0(Y)/Pic^0(Y)_{\Phi(\OO_{X}(m))}$ is an abelian variety isogenous to $Pic^0(Y)$, which embeds into $\Phi(Pic^{0,m}(X))$, hence into $Pic^0(X)$.

Using the inverse equivalence $\Phi^{-1}$, we get an embedding of an abelian variety isogeneous to $Pic^0(X)$ into $Pic^0(Y)$. This finally proves that $Pic^0(X)$ and $Pic^0(Y)$ are isogeneous.

\bigskip

Let me turn to the proof of $(2)$. We denote again by $\Phi$ the equivalence between $\DB(X)$ and $\DB(Y)$ and also by $\Phi$ the induced isomorphism between the derived stacks $DPerf(X)$ and $DPerf(Y)$. In the course of the proof of item $(5)$ of Theorem \ref{douce}, I proved that there exists a line bundle $L_0 \in Pic(X)$ such that the rank of $\Phi(L_0)$ as a bounded complex of sheaves on $Y$ is non-zero. We denote by $Pic(X)^{0,L_0}_{dg}$ the canonical torsor under $Pic^0(X)_{dg}$, with base point $L_0$. 

\bigskip

 Let $f : Pic^0(Y)_{dg} \rightarrow \Phi(Pic^{0,L_0}(X)_{dg})$ be the map of derived stacks defined by $f(M) = \Phi(L_0) \otimes M$. This map is well-defined, as $\Phi(Pic^{0,L_0}(X)_{dg})$ is a connected component of $DPerf(Y)$. The tangent complex of $Pic^0(Y)_{dg}$ at any $M \in Pic^0(Y)_{dg}$ is $\mathrm{Ext}_Y^{\bullet}(M, M)[-1] $ and the tangent complex to $\Phi(Pic(X)^{0,L}_{dg})$ at any $\Phi(L)$ is $\mathrm{Ext}_Y^{\bullet}(\Phi(L), \Phi(L))[-1]$ (see \cite{TV}). Furthermore, since $f$ is given by tensor product, for any $M \in Pic^{0}(Y)_{dg}$, the map:

\begin{equation*}
T_{f} : T_{Pic^0(Y)_{dg},M} \simeq \mathrm{Ext}_Y^{\bullet}(M,M)[-1] \rightarrow T_{\Phi(Pic(X)^{0,L}_{dg}), \Phi(L_0) \otimes M} \simeq \mathrm{Ext}_Y^{\bullet}(\Phi(L_0)\otimes M, \Phi(L_0) \otimes M)[-1]
\end{equation*}
is the shifted dual of the trace map:

\begin{equation*}
Tr : \mathrm{Ext}_Y^{\bullet}(\Phi(L_0)\otimes M, \Phi(L_0) \otimes M) \rightarrow \mathrm{Ext}_Y^{\bullet}(M,M).
\end{equation*}
Since the rank of $\Phi(L_0) \otimes M$ is not zero, the map $T_{f}$ induces an injection in cohomology. But by item $(5)$ of Theorem \ref{douce}, we know that the cohomology groups of $\mathrm{Ext}_Y^{\bullet}(M,M)$ and $\mathrm{Ext}_Y^{\bullet}(\Phi(L_0)\otimes M, \Phi(L_0) \otimes M)$ are isomorphic. Hence $T_f$ is a quasi-isomorphism of DG-algebra, so that the map $f$ is \'etale in the derived sense. By $(1)$, we already know that the $0-th$-truncated part of $f$ is an isogeny of abelian varieties. We deduce that $f$ is an isogeny of derived Jacobians. This concludes the proof of the Theorem.

\end{proof}

\bigskip

The thoughtful reader has noticed that to prove $(1)$ of Theorem \ref{jacjac}, we only need the base field to be algebraically closed of characteristic $0$. On the other hand, in order to prove item $(2)$ of Theorem \ref{jacjac}, we rely on item $(5)$ of Theorem \ref{douce}, which in turn is based on the Hodge-Riemann bilinear relations. However, one can not seriously believe that isogeny of the (classical) Jacobians of derived equivalent smooth projective varieties could be proved on any algebraically closed field of characteristic $0$, while isogeny of the derived Jacobians would require transcendental methods. 

\bigskip

This suggests that a proof of item $(5)$ in Theorem \ref{douce} not based on transcendental methods should be found. Once such a proof will have been discovered, it will be certainly possible to demonstrate the invariance of homological units in any dimension. Then, the isogeny of the derived Jacobians of derived equivalent smooth projective varieties would follow using exactly the same arguments as in the proof of item $(2)$ of Theorem \ref{jacjac} 
\end{section}

\newpage

\small{
\bibliographystyle{alpha}

\bibliography{bibliHKC}}
\end{document}